\documentclass[11pt]{amsart}

\usepackage{amsmath,amsthm,amssymb,dsfont,enumerate,tikz,hyperref}
\usepackage{setspace}
\usepackage[margin=1.25in]{geometry}

\theoremstyle{theorem}
\newtheorem{theorem}{Theorem}[section]
\newtheorem{algorithm}[theorem]{Algorithm}
\newtheorem{proposition}[theorem]{Proposition}

\theoremstyle{definition}
\newtheorem{example}[theorem]{Example}
\newtheorem{remark}[theorem]{Remark}

\numberwithin{equation}{section}
\numberwithin{figure}{section}

\DeclareMathOperator*{\argmin}{arg\,min}

\renewcommand{\subset}{\subseteq}

\renewcommand{\hat}{\widehat}
\renewcommand{\epsilon}{\varepsilon}
\def\sign{\text{sign}}
\def\supp{\text{supp}}

\def\<{\langle}
\def\>{\rangle}
\def\({\Big(}
\def\){\Big)}
\def\C{\mathbb{C}}
\def\E{\mathcal{E}}
\def\F{\mathcal{F}}

\def\Q{\mathbb{Q}}
\def\R{\mathbb{R}}
\def\T{\mathbb{T}}
\def\Z{\mathbb{Z}}

\title[Beurling super-resolution]{Super-resolution by means of Beurling minimal extrapolation}
\author{John J. Benedetto \and Weilin Li}
\date{}



\address{Norbert Wiener Center, Department of Mathematics, University of Maryland, College Park, MD 20742, USA}
\email{jjb@math.umd.edu}

\address{Norbert Wiener Center, Department of Mathematics, University of Maryland, College Park, MD 20742, USA}
\email{wl298@math.umd.edu}
\begin{document}

\begin{abstract}
	Let $M(\T^d)$ be the space of complex bounded Radon measures defined on the $d$-dimensional torus group 
	$(\R/\Z)^d=\T^d$, equipped with the total variation norm $\|\cdot\|$; and let $\hat\mu$ denote the Fourier 
	transform of $\mu\in M(\T^d)$. We address the super-resolution problem: For given spectral (Fourier transform) 
	data defined on a finite set $\Lambda\subset\Z^d$, determine if there is a unique $\mu\in M(\T^d)$ of minimal norm for which 
	$\hat\mu$ equals this data on $\Lambda$. Without additional assumptions on $\mu$ and $\Lambda$, our main 
	theorem shows that the solutions to the super-resolution problem, which we call \emph{minimal extrapolations}, 
	depend crucially on the set $\Gamma\subset\Lambda$, defined in terms of $\mu$ and $\Lambda$. For example, 
	when $\#\Gamma=0$, the minimal extrapolations are  singular measures supported in the zero set of an analytic 
	function, and when $\#\Gamma\geq 2$, the minimal extrapolations are singular measures supported in the 
	intersection of $\#\Gamma\choose 2$ hyperplanes. By theory and example, we show that the case $\#\Gamma=1$ 
	is different from other cases and is deeply connected with the existence of positive minimal extrapolations. 
	This theorem has implications to the possibility \emph{and} impossibility of uniquely recovering $\mu$ from 
	$\Lambda$. We illustrate how to apply our theory to both directions, by computing pertinent analytical examples. 
	These examples are of interest in both super-resolution and deterministic compressed sensing. Our concept of an \emph{admissibility range} fundamentally connects Beurling's theory of minimal extrapolation \cite{beurling1989balayage,beurling1989interpolation} with Cand\`{e}s and Fernandez-Granda's work 
	on super-resolution \cite{candes2014towards}. This connection is exploited to address situations 
	where current algorithms fail to compute a numerical solution to the super-resolution problem.
\end{abstract}

\maketitle

\section{Introduction}

\subsection{Motivation}
The term \emph{super-resolution} varies depending on the field, and, consequently, there are various types of super-resolution problems. In some situations \cite{park2003super}, super-resolution refers to the process of up-sampling an image onto a finer grid, which is a \emph{spatial interpolation} procedure. In other situations \cite{lindberg2012mathematical}, super-resolution refers to the process of recovering the object's high frequency information from its low frequency information, which is a \emph{spectral extrapolation} procedure. In both situations, the super-resolution problem is ill-posed because the missing information can be arbitrary. However, it is possible to provide meaningful super-resolution algorithms by using \emph{prior knowledge} of the data. We develop a mathematical theory of the spectral extrapolation version, which we simply refer to as the \emph{super-resolution problem}, see Problem (\ref{SR}) for a precise statement. 

To be concrete, our exposition focuses on imaging applications, but super-resolution ideas are of interest in other fields, e.g., \cite{rieke1999spikes,khaidukov2004diffraction}. In such applications, an image is obtained by convolving the object with the \emph{point spread function} of an optical lens. Alternatively, the Fourier transform of the object is multiplied by a \emph{modulation transfer function}. The resulting image's resolution is inherently limited by the Abbe diffraction limit, which depends on the illumination light's wavelength and on the diameter of the optical lens. Thus, the optical lens acts as a low-pass filter, see \cite{lindberg2012mathematical}. The purpose of super-resolution is to use prior knowledge about the object to obtain an accurate image whose resolution is higher than what can be measured by the optical lens.

We mention two specific imaging problems, and they motivate the theory that we develop. 
\begin{enumerate}[(a)]
	\item 
	In astronomy \cite{puschmann2005super}, each star can be modeled as a complex number times a Dirac $\delta$-measure, and the Fourier transform of each star encodes important information about that star. However, an image of two stars that are close in distance resembles an image of a single star. In this context, the super-resolution problem is to determine the number of stars and their locations, using the prior information that the actual object is a linear combination of Dirac $\delta$-measures. 
	
	\item
	In medical imaging \cite{greenspan2009super}, machines capture the structure of the patient's body tissues, in order to detect for anomalies in the patient. Their shapes and locations are the most important features, so each tissue can be modeled as the characteristic function of a closed set, or as a surface measure supported on the boundary of a set. The super-resolution problem is to capture the fine structures of the tissues, given that the actual object is a linear combination of singular measures. 
\end{enumerate}

\subsection{Problem statement}

We model objects as elements of $M(\T^d)$, the space of complex bounded Radon measures on $\T^d=(\R/\Z)^d$, the $d$-dimensional torus group. $M(\T^d)$ equipped with the total variation norm $\|\cdot\|$ is a Banach algebra with unit, the Dirac $\delta$-measure, where multiplication is defined by convolution $\ast$.  
The \emph{Fourier transform} of $\mu\in M(\T^d)$ is the function $\hat{\mu}\colon\Z^d\to\C$, whose $m$-th Fourier coefficient is defined as
\[
\hat{\mu}(m)=\int_{\T^d} e^{-2\pi im\cdot x}\ d\mu(x).
\]
See \cite{benedetto2010integration} for further details.

We consider the \emph{super-resolution problem}: For a given finite subset $\Lambda\subset\Z^d$ and 
given spectral data 
$F$ of the form $\hat\mu\mid_\Lambda$ for some unknown $\mu\in M(\T^d)$, find 
\begin{equation}
\label{SR}
\tag{SR}
\argmin_{\nu\in M(\T^d)} \|\nu\|
\quad \text{subject to}\quad
\hat{\nu}=\hat{\mu} 
\quad\text{on } \Lambda.
\end{equation}
This is a is a \emph{convex minimization problem} and we interpret a solution as a  \emph{simple} or 
\emph{least complicated} high resolution extrapolation of $F.$

\begin{remark}[Image processing]
	\begin{enumerate} [(a)]
		\item 
		A primary objective is to determine if $\mu\in M(\T^d)$ is a solution to the super-resolution problem given its spectral data $\hat\mu\mid_\Lambda$. The current literature has focused on discrete $\mu\in M(\T^d)$, e.g., see \cite{candes2014towards,tang2013compressed,de2012exact}. However, in the context of image processing, $\mu$ is the unknown high resolution image, and a typical image cannot be approximated by a discrete measure. Hence, it is important to determine if non-discrete measures are solutions to the super-resolution problem. For this reason, we have formulated and shall study Problem (\ref{SR}) in an abstract way. 
	
		\item
		Further, in the context of image processing, we can think of being given information that represents an image 
		$\mu\in M(\T^d),$ and this information is in the form, 
		\[
		F(x)
		=(\mu*\psi)(x)
		=\int_{\T^d} \psi(x-y)\ d\mu(y),
		\]
		for some $\psi\colon\T^d\to\C$. For simplicity, we assume that $\hat\psi=\mathds{1}_\Lambda$, the characteristic function of some finite set $\Lambda\subset\Z^d$. Then, we have
		\[
		F
		=(\mu*\psi)
		=(\hat\mu\mid_\Lambda)^\vee.
		\]
		Thus, we interpret $\hat\mu\mid_\Lambda$ as the given low frequency information of $\mu$, and $(\hat{\mu}\mid_\Lambda)^\vee$ as the given low resolution image, even though in applications, we do not know the desired object $\mu$. This is the background for our formulation of Problem (\ref{SR}).
\end{enumerate}
\end{remark}
		
With regard to Problem (\ref{SR}), a fundamental question of uniqueness must be addressed. To see why this is so, if $\mu$ is not the unique solution to the super-resolution problem, then the output of a numerical algorithm is not guaranteed to approximate $\mu$. Thus, it is important to determine sufficient conditions such that $\mu$ is the unique solution. For this reason, we say that \emph{super-resolution reconstruction of $\mu$ from $\Lambda$ is possible}, or \emph{it is possible to super-resolve $\mu$ from $\Lambda$}, if and only if $\mu$ is the unique solution to Problem (\ref{SR}). Of course, it could be theoretically be possible to reconstruct $\mu$ by other means.

\subsection{Background}

Problem (\ref{SR}) was first studied by Cand\`{e}s and Fernandez-Granda \cite{candes2014towards} and was inspired by results from compressed sensing \cite{donoho2006compressed,candes2006robust}. By considering different models of images and/or alternative measurement processes, it is possible to derive related (noiseless) ``super-resolution" problems, e.g., see \cite{donoho1992superresolution, de2012exact, tang2013compressed, aubel2014super, aubel2015theory}. There are several papers that address ``super-resolution" in the presence of noise, e.g., see \cite{candes2013super, bhaskar2013atomic, azais2015spike, tang2015near, duval2015exact}. 

We briefly discuss the important results of Cand{\'e}s and Fernandez-Granda \cite{candes2014towards}. Given $\Lambda = \{-M,-M+1, \dots, M\}^d\subset\Z^d$, we say $\mu=\sum_{k=1}^K a_k\delta_{x_k}\in M(\T^d)$ satisfies the \emph{minimum separation condition with constant $C>0$} if 
\begin{equation}
\label{separation}
\inf_{\substack{1\leq j,k\leq K \\ j\not=k}}\|x_j-x_k\|_{\ell^\infty(\T^d)}
\geq \frac{C}{M}.
\end{equation}

\begin{theorem}[Cand{\`e}s and Fernandez-Granda, Theorems 1.2-1.3, \cite{candes2014towards}. Fernandez-Granda, Theorem 2.2, \cite{fernandez2016super}]
	\label{thm CFG}
	Let $\mu=\sum_{k=1}^K a_k\delta_{x_k}\in M(\T^d)$ and $\Lambda=\{-M,-M+1,\dots,M\}^d\subset\Z^d$. 
	\begin{enumerate}[(a)]
		\item 
		If $d=1$, $M\geq 128$, and $\mu$ satisfies the minimum separation condition with $C=2$, then $\mu$ is the unique solution to Problem (\ref{SR}). If additionally, $\mu$ is real valued, then the constant can be reduced to $C=1.87$.
		\item
		If $d=1$, $M\geq 10^3$, and $\mu$ satisfies the minimum separation condition with $C=1.26$, then $\mu$ is the unique solution to Problem (\ref{SR}).
		\item
		If $d=2$, $\mu$ is real valued, $M\geq 512$, and $\mu$ satisfies the minimum separation condition for $C=2.38$, then $\mu$ is the unique solution to Problem (\ref{SR}).
	\end{enumerate}
\end{theorem}

The theorem proves that, in theory, it is possible to uniquely recover discrete measures satisfying a sufficiently strong minimum separation condition. To compute a numerical approximation of such a measure, they proposed a convex-optimization based algorithm which relies on \cite[Theorem 4.24]{dumitrescu2007positive}. They stated their algorithm for $d=1$, but it is possible an analogous algorithm for higher dimensions. 

\begin{algorithm} [Cand\'{e}s and Fernandez-Granda, Section 4, \cite{candes2014towards}] \label{alg CFG}
	Suppose $\mu\in M(\T)$ and $\Lambda=\{-M,-M+1,\dots,M\}\subset\Z$ satisfy the hypotheses in any of the three situations described in Theorem \ref{thm CFG}.
	\begin{enumerate}
		\item 
		Solve the convex optimization problem described in \cite[Corollary 4.1]{candes2014towards} to obtain a function $\varphi(x)=\sum_{m=-M}^M c_me^{2\pi im x}$ such that $\|\varphi\|_\infty \leq 1$ and $\<\varphi,\mu\>=\|\mu\|$. Then, $\mu$ is supported in the set,
		\[
		S=\{x\in\T\colon |\varphi(x)|=1\}.
		\] 
		
		\item
		Assume that $S\not=\T$. Then, $S$ is a discrete set that consists of at most $2M$ points, and it is possible to determine $\mu$ by solving a system of over-determined equations.
	\end{enumerate}	
\end{algorithm}

Cand\`{e}s and Fernandez-Granda pointed out that their algorithm for recovering $\mu$ is incomplete because it fails precisely when $\{x\in\T\colon |\varphi(x)|=1\}=\T$, i.e., when $|\varphi|\equiv 1$, since in this case, $\varphi$ does not provide information about the support of $\mu$. We emphasize that this scenario can occur even if $\mu\in M(\T)$ and $\Lambda\subset\Z$ satisfy the hypotheses of Theorem \ref{thm CFG}. 

\subsection{Our approach}
\label{section approach}

This paper connects the Cand\'{e}s and Fernandez-Granda theory on super-resolution \cite{candes2014towards} with Beurling's work on minimal extrapolation \cite{beurling1989balayage, beurling1989interpolation}. This connection is exploited to obtain new theoretical and computational results on super-resolution.

We first introduce some additional notation, since referring to Problem (\ref{SR}) can be ambiguous when working with several different measures. We adopt Beurling's language, and we say that $\nu\in M(\T^d)$ is a \emph{minimal extrapolation} of $\mu\in M(\T^d)$ from $\Lambda\subset\Z^d$ if it is a solution to Problem (\ref{SR}). If $\nu\in M(\T^d)$ and $\hat\mu=\hat\nu$ on $\Lambda$, then we say $\nu$ is an \emph{extrapolation} of $\mu$ from $\Lambda$. When there is no ambiguity, we shall say that ``$\nu$ is an extrapolation (respectively, a minimal extrapolation)" as a shortened version of ``$\nu$ is an extrapolation (respectively, a minimal extrapolation) of $\mu$ from $\Lambda$".

Let $\E=\E(\mu,\Lambda)$ be the set of all minimal extrapolations. For any $\nu\in\E$, let $\epsilon=\epsilon(\mu,\Lambda)=\|\nu\|$, which is the optimal value attained by Problem (\ref{SR}). It is possible to compute a numerical approximation of $\epsilon$ by solving the convex optimization problem described in \cite[Coroallary 4.1]{candes2014towards}, but its exact value is unknown. Finally, we define the set,
\begin{equation}
\label{Gamma}
\Gamma
=\Gamma(\mu,\Lambda)
=\{m\in\Lambda\colon |\hat\mu(m)|=\epsilon(\mu,\Lambda)\}.
\end{equation}
Although $\Gamma$ is defined in terms of the unknown quantity $\epsilon$, we shall see that, for many important applications, it is possible to deduce $\epsilon$ and $\Gamma$. Our theorem shows that the minimal extrapolations intimately depend on $\Gamma$. 

\begin{theorem}
	\label{thm BL}
	Let $\mu\in M(\T^d)$ and let $\Lambda\subset\Z^d$ be a finite subset. 
	\begin{enumerate}[(a)]
		\item
		Suppose $\#\Gamma=0$. Then, there exists a closed set $S$ of $d$-dimensional Lebesgue measure zero such that each minimal extrapolation is a singular measure supported in $S$. In particular, if $d=1$, then $S$ is a finite number of points and each minimal extrapolation is a discrete measure supported in $S$.
		\item
		Suppose $\#\Gamma\geq 2$. For each distinct pair $m,n\in\Gamma$, define $\alpha_{m,n}\in\R/\Z$ by $e^{2\pi i\alpha_{m,n}}=\hat\mu(m)/\hat\mu(n)$ and the closed set,
		\begin{equation}
		\label{S}
		S=\bigcap_{\substack{m,n\in\Gamma \\ m\not=n}} \{x\in\T^d \colon x\cdot (m-n)+\alpha_{m,n}\in\Z\},
		\end{equation}
		which is an intersection of $\#\Gamma\choose 2$ periodic hyperplanes. Then, each minimal extrapolation is a singular measure supported in $S$. In particular, if $d=1$, then $S$ is a finite number of points and each minimal extrapolation is a discrete measure supported in $S$.
		
		If $d\geq 2$ and there exist $d$ linearly independent vectors, $p_1,p_2,\dots,p_d\in\Z^d$, such that
		\[
		\{p_1,p_2,\dots,p_d\}\subset\{m-n\colon m,n\in\Gamma\},
		\]
		then $S$ is a lattice on $\T^d$.  
	\end{enumerate}
\end{theorem}

The following topics encapsulate the contributions of this paper. 

\begin{enumerate}[(a)]
	
	\item
	\emph{Qualitative behavior of minimal extrapolations}. It is interesting to note that $\Gamma$ provides 
significant information about the minimal extrapolations.  Theorem \ref{thm BL} shows that, regardless of the dimension, when $\#\Gamma=0$ or $\#\Gamma\geq 2$, the minimal extrapolations are always singular measures. However, when $\#\Gamma=1$, the minimal extrapolations can be extremely different from each other. Proposition \ref{prop5} shows that the case $\#\Gamma=1$ is connected with the existence of positive absolutely continuous minimal extrapolations. In Example \ref{e2}, we provide an example for which there are uncountably many discrete minimal extrapolations as well as uncountably many positive absolutely continuous minimal extrapolations.   
	
	\item
	\emph{Computational consequences}. One important aspect of Theorem \ref{thm BL} is its relationship with Algorithm \ref{alg CFG}. Proposition \ref{prop4} shows that, if the algorithm fails, then we necessarily have $\epsilon=\|\hat\mu\|_{\ell^\infty(\Lambda)}$ and $\Gamma\not=\emptyset$. Since we are given the values of $\hat\mu$ on $\Lambda$, this immediately tells us what $\epsilon$ and $\Gamma$ are. If $\#\Gamma=1$, then we cannot deduce anything about the minimal extrapolations. As previously discussed, this is inherent with the super-problem and is not an artifact of our analysis. If $\#\Gamma\geq 2$, then the theorem is applicable and the minimal extrapolations are singular measures supported in the set defined in (\ref{S}), which can be explicitly computed. We show how to apply our theorem to compute pertinent analytical examples in Section \ref{examples}. Hence, our theorem is applicable even when Algorithm \ref{alg CFG} fails. 
	
	\item
	\emph{Super-resolution of singular measures}. When $d\geq 2$, Theorem \ref{thm BL} suggests that some singular continuous measures could be solutions to the super-resolution problem. Somewhat surprisingly, this is indeed the case because Example \ref{e5} provides an example of a singular continuous minimal extrapolation. This also demonstrates that the conclusion of Theorem \ref{thm BL}b is optimal. Determining what types of singular continuous measures can be uniquely recovered by Problem (\ref{SR}) is an interesting mathematical problem, and we believe such a result would be useful in applications, e.g., \cite{greenspan2009super}. We shall carefully address the super-resolution of singular continuous measures in the sequel \cite{benli2}. 
	
	\item
	\emph{Impossibility of super-resolution}. Theorem \ref{thm BL} does not require additional assumptions on $\mu\in M(\T^d)$ or on the finite subset $\Lambda\subset\Z^d$. Since the theorem also describes the support set of the minimal extrapolations of $\mu$ from $\Lambda$, it is useful for determining whether it is possible for a given $\mu$ to be a minimal extrapolation. To illustrate this point, in Example \ref{e6}, we provide a simple proof that shows, in general, a minimal separation condition is necessary to super-resolve a discrete measure. 

	\item
	\emph{Discrete-continuous correspondence}. The super-resolution problem can be viewed as a continuous analogue of the discrete Fourier compressed sensing problem that was studied in \cite{candes2006robust}. Let $\F_N\in\C^{N\times N}$ be the DFT matrix, $x\in\C^N$ be an unknown vector, and $\Lambda\subset\{0,1,\dots,N-1\}$. The goal is to recover $x$ from $\F_Nx\mid_\Lambda$ by solving
	\begin{equation}
	\label{CS}
	\tag{CS}
	\argmin_{y\in\C^N}\|y\|_{\ell^1} 
	\quad\text{subject to}\quad
	\F_N x=\F_N y 
	\quad\text{on }	\Lambda.
	\end{equation} 
	By considering the case that $\mu=\sum_{n=1}^N a_n\delta_{n/N}$ and $x(n)=a_n$, it is straightforward to see that Problem (\ref{SR}) is a generalization of Problem (\ref{CS}), see \cite{candes2014towards} for further details. From this point of view, Theorem \ref{thm BL} is a discrete-continuous correspondence result. Indeed, if $\#\Gamma\geq 2$, and either $d=1$ or the vectors $\{p_j\}$ exist, then the minimal extrapolations are necessarily discrete measures whose support lie on a lattice; as we just discussed, such measures can be identified with vectors that are solutions to the discrete problem.   
	
	\item
	\emph{Mathematical connections}. Our results are closely related to Beurling's work on minimal extrapolation. He dealt with $\R^1$ instead of $\T^d$, so Theorem \ref{thm BL} is an adaptation to the torus and a generalization to higher dimensions of \cite[Theorem 2, page 362]{beurling1989interpolation}. We remark that there are non-trivial technical differences between working with $\R^1$ and $\T^d$. There are also other classical papers related to minimal extrapolation, e.g., \cite{esseen1945fourier,esseen1954note,hewitt1955rs}. 
	
	Propositions \ref{prop3} connects our results to the Cand\`{e}s and Fernandez-Granda theory \cite{candes2014towards} by introducing a concept called an admissibility range for $\epsilon$. This connection is exploited to prove Proposition \ref{prop4}, which in turn, shows that our theorem is applicable to situations where Algorithm \ref{alg CFG} fails.
	
	Proposition \ref{prop1} contains basic duality results for Problem (\ref{SR}) that can be derived from abstract convex analysis, e.g., see \cite{duval2015exact}. None of our proofs require knowledge of convex analysis, so we believe our paper is accessible to non-specialists. 
	
\end{enumerate}

\subsection{Outline}
Section \ref{section theory} contains the basis of our mathematical theory. Section \ref{section functional} proves well-known results on the super-resolution problem using only basic functional analysis. Section \ref{section Gamma} provides a characterization of $\Gamma$ and illustrates its importance in our analysis. The proof of Theorem \ref{thm BL} is given in Section \ref{section Beurling}. Section \ref{section admissibility} contains the results that fully connect the Beurling and Cand\`{e}s and Fernandez-Granda theories, as well as the the relationship between Theorem \ref{thm BL} and Algorithm \ref{alg CFG}. Section \ref{section uniqueness} includes a basic uniqueness result, and the mentioned non-uniqueness results for the case $\#\Gamma=1$. Section \ref{section basic} examines whether it is possible to relate the minimal extrapolations of $\mu$ with those of $T\mu$ for various operators $T$. Section \ref{examples} contains all of the examples that we have previously mentioned, including the minimal extrapolations of discrete measures and of singular continuous measures. 

\section{Mathematical Theory}
\label{section theory}

\subsection{Preliminary results}
\label{section functional}

Let $C(\T^d)$ be the space of complex-valued continuous functions on $\T^d$ equipped with the sup-norm $\|\cdot\|_\infty$. Then, $C(\T^d)$ is a Banach space, and let $C(\T^d)'$ be its dual space of continuous linear functionals with the usual norm $\|\cdot\|$. The celebrated Riesz representation theorem, e.g., see \cite[Theorem 7.2.7, page 334]{benedetto2010integration} states that:
\begin{enumerate}[(a)]
	\item 
	For each $\mu\in M(\T^d)$, there exists a bounded linear functional $\ell_\mu\in C(\T^d)'$ such that $\|\mu\|=\|\ell_\mu\|$ and 
	\[
	\forall f\in C(\T^d), \quad
	\ell_\mu(f)
	=\<f,\mu\>
	=\int_{\T^d} f(x)\ \overline{d\mu(x)}.
	\]
	\item
	For each bounded linear functional $\ell\in C(\T^d)'$, there exists a unique $\mu\in M(\T^d)$ such that $\|\mu\|=\|\ell\|$ and
	\[
	\forall f\in C(\T^d), \quad
	\ell(f)
	=\<f,\mu\>
	=\int_{\T^d} f(x)\ \overline{d\mu(x)}.
	\]
\end{enumerate}

Proposition \ref{prop1}a shows that the super-resolution problem is well-posed, by using standard functional analysis arguments. However, this type of argument does not yield useful statements about the minimal extrapolations. Instead of working with $C(\T^d)$, we shall work with a subspace. Since only $\hat\mu\mid_\Lambda$ is important to the super-resolution problem, we consider the subspace  
\[
C(\T^d;\Lambda)
=\Big\{f\in C(\T^d)\colon f(x)=\sum_{m\in\Lambda} a_me^{2\pi im\cdot x},\ a_m\in\C\Big\}.
\]
Proposition \ref{prop1}b shows that $C(\T^d;\Lambda)$ is a closed subspace of $C(\T^d)$, and that implies $C(\T^d;\Lambda)$ is a Banach space. Further, Proposition \ref{prop1}c demonstrates that
\[
U=U(\T^d;\Lambda)=\{f\in C(\T^d;\Lambda)\colon \|f\|_\infty\leq 1\},
\]
the closed unit ball of $C(\T^d;\Lambda)$, is compact. 

The purpose of restricting to the subspace, $C(\T^d;\Lambda)$, is to identify $\mu\in M(\T^d)$ with the bounded linear functional, $L_\mu\in C(\T^d;\Lambda)'$, defined as
\[
\forall f\in C(\T^d;\Lambda),\quad
L_\mu(f)
=\int_{\T^d} f(x)\ \overline{d\mu(x)}.
\]
Although, by definition, $\|L_\mu\|=\sup_{f\in U}|L_\mu(f)|$, Proposition \ref{prop1}d,e show that we have the stronger statement,
\[
\|L_\mu\|=\max_{f\in U}|L_\mu(f)|=\epsilon. 
\]

The purpose of studying $L_\mu$ is to deduce information about the minimal extrapolations. As a consequence of the Radon-Nikodym theorem, for each $\mu\in M(\T^d)$, there exists a $\mu$-measurable function $\sign(\mu)$ such that $|\sign(\mu)|=1$ $\mu$-a.e. and satisfies the identity
\[
\forall f\in L^1_{|\mu|}(\T^d),\quad 
\int_{\T^d} f\ d|\mu|=\int_{\T^d} f\ \overline{\sign(\mu)}\ d\mu.
\]
See \cite[Theorem 5.3.2, page 242, and Theorem 5.3.5, page 244]{benedetto2010integration} for further details. The support of $\mu\in M(\T^d)$, denoted $\supp(\mu)$, is the complement of all open sets $A\subset\T^d$ such that $\mu(A)=0$. Proposition \ref{prop1}g shows that it is possible to interpolate the sign of the minimal extrapolations by a function belonging to $U$.

\begin{proposition}
	\label{prop1}
	Let $\mu\in M(\T^d)$ and let $\Lambda\subset\Z^d$ be a finite subset.
	\begin{enumerate}[(a)]
	\item
	$\E\subset M(\T^d)$ is non-empty, weak-$\ast$ compact, and convex.
	\item 
	$C(\T^d;\Lambda)$ is a closed subspace of $C(\T^d)$.
	\item
	$U$ is a compact subset of $C(\T^d;\Lambda)$.
	\item
	$\epsilon=\|L_\mu\|$.
	\item
	$\epsilon=\max_{f\in U}|\<f,\mu\>|$.
	\item
	There exists $\varphi\in U$ such that $\<\varphi,\mu\>=\epsilon$.
	\item
	If $\varphi\in U$ and $\<\varphi,\mu\>=\epsilon$, then
	\[
	\forall\nu\in\E,\quad 
	\varphi=\sign(\nu) \quad
	\nu \text{-a.e.,}\quad \text{and}\quad 
	\supp(\nu)\subset\{x\in\T^d\colon |\varphi(x)|=1\}.
	\]
	\end{enumerate} 
\end{proposition}

\begin{proof}
	\indent
	\begin{enumerate}[(a)]
	\item
	By definition of $\epsilon$, there exists a sequence $\{\nu_j\}\subset\{\nu\colon \hat{\nu}=\hat{\mu}\text{ on }\Lambda\}$ such that $\|\nu_j\|\to\epsilon$. Then, this sequence is bounded. By Banach-Alaoglu, after passing to a subsequence, we can assume there exists $\nu\in M(\T^d)$ such that $\nu_j\to\nu$ in the weak-$\ast$ topology. 
	
	Let $V$ be the closed unit ball of $C(\T^d)$. We have $\|\nu\|\leq\epsilon$ because
	\[
	\|\nu\|
	=\sup_{f\in V} |\<f,\nu\>|
	=\sup_{f\in V} \lim_{j\to\infty} |\<f,\nu_j\>|
	\leq \sup_{f\in V} \lim_{j\to\infty} \|f\|_\infty \|\nu_j\|
	\leq \epsilon.
	\]
	Moreover, for $m\in\Lambda$, we have
	\[
	\hat{\mu}(m)
	=\lim_{j\to\infty}\hat{\nu_j}(m)
	=\lim_{j\to\infty}\int_{\T^d} e^{-2\pi im\cdot x}\ d\nu_j(x)
	=\int_{\T^d} e^{-2\pi im\cdot x}\ d\nu(x)
	=\hat{\nu}(m).
	\]
	This shows that $\nu$ is an extrapolation, and thus, $\|\nu\|\geq\epsilon$. Therefore, $\nu\in\E$.
	
	The proof that $\E$ is weak-$\ast$ compact is similar. Pick any sequence $\{\nu_j\}\subset\E$ and after passing to a subsequence, we can assume $\nu_j\to \nu$ in the weak-$\ast$ topology for some $\nu\in M(\T^d)$. By the same argument, we see that $\nu\in\E$.
	
	
	If $\E$ contains exactly one measure, then $\E$ is trivially convex. Otherwise, let $t\in[0,1]$, $\nu_0,\nu_1\in\E$, and $\nu_t=(1-t)\nu_0+t\nu_1$. Then, $\nu_t$ is an extrapolation and thus, $\|\nu_t\|\geq \epsilon$. By the triangle inequality, we have $\|\nu_t\|\leq(1-t)\|\nu_0\|+t\|\nu_1\|=\epsilon$. Thus, $\nu_t\in\E$ for each $t\in[0,1]$. 
	
	\item 
	Suppose $\{f_j\}\subset C(\T^d)$ and that there exists $f\in C(\T^d)$ such that $f_j\to f$ uniformly. Then, $\hat{f_j}(m)\to\hat f(m)$ for all $m\in \Z^d$. Since $\hat{f_j}(m)=0$ if $m\not\in\Lambda$, we deduce that $\hat f(m)=0$ if $m\not\in\Lambda$. This shows that $f\in C(\T^d;\Lambda)$ and thus, $C(\T^d;\Lambda)$ is a closed subspace of $C(\T^d)$.  
	
	\item
	Let $\{f_j\}\subset U$. We first show that $\{f_j\}$ is a uniformly bounded equicontinuous family. By definition, $\|f_j\|_\infty\leq 1$, and there exist $a_{j,m}\in\C$ such that $f_j(x)=\sum_{m\in\Lambda}a_{j,m}e^{2\pi im\cdot x}$. Note that $|a_{j,m}|\leq\|f_j\|_{\infty}\leq 1$. Then, for any $x,y\in\T^d$, we have
	\[
	|f_j(x)-f_j(y)|
	=\Big|\sum_{m\in\Lambda} a_{j,m} (e^{2\pi im\cdot x}-e^{2\pi im\cdot y})\Big|
	\leq \sum_{m\in\Lambda} |e^{2\pi im\cdot x}-e^{2\pi im\cdot y}|. 
	\]
	Let $\epsilon>0$ and $m\in\Lambda$. There exists $\delta_m>0$ such that $|e^{2\pi im\cdot x}-e^{2\pi im\cdot y}|<\epsilon$ whenever $|x-y|<\delta_m$. Let $\delta=\min_{m\in\Lambda}\delta_m$. Combining this with the previous inequalities, we have 
	\[
	|f_j(x)-f_j(y)|
	\leq  \sum_{m\in\Lambda} |e^{2\pi im\cdot x}-e^{2\pi im\cdot y}|
	\leq \#\Lambda \epsilon,
	\]
	whenever $|x-y|<\delta$. This shows that $\{f_j\}$ is a uniformly bounded equicontinuous family. 
	
	By the Arzel\`{a}-Ascoli theorem, there exists $f\in C(\T^d)$, with $\|f\|_\infty\leq 1$, and a subsequence $\{f_{j_k}\}$ such that $f_{j_k}\to f$ uniformly. Thus, $\hat{f_{j_k}}(m)\to \hat f(m)$ for all $m\in\Z^d$, which shows that $f\in C(\T^d;\Lambda)$. 
	
	\item
	Let $\nu\in\E$. Then,
	\[
	\forall f\in U,\quad 
	|L_\mu(f)|
	=|\<f,\mu\>|
	=|\<f,\nu\>|
	\leq \|f\|_\infty\|\nu\|
	\leq\epsilon,
	\]
	which proves the upper bound, $\|L_\mu\|\leq\epsilon$.
	
	For the lower bound, we use the Hahn-Banach theorem to extend $L_\mu\in C(\T^d;\Lambda)'$ to $\ell\in C(\T^d)'$, where $\|L_\mu\|=\|\ell\|$. By the Riesz representation theorem, there exists a unique $\sigma\in M(\T^d)$ such $\ell(f)=\<f,\sigma\>$ for all $f\in C(\T^d)$ and $\|\sigma\|=\|\ell\|$. In particular, 
	\[
	\forall f\in C(\T^d;\Lambda), \quad
	\<f,\sigma\>=\ell(f)=L_\mu(f)=\<f,\mu\>.
	\]
	Set $f(x)=e^{-2\pi im\cdot x}$, where $m\in\Lambda$, to deduce that $\hat\mu=\hat\sigma$ on $\Lambda$. This implies $\|\sigma\|\geq\epsilon$. Combining these facts, we have
	\[
	\epsilon
	\leq\|\sigma\|
	=\|\ell\|
	=\|L_\mu\|.
	\]
	This proves the lower bound. 
	
	\item
	We know that $\|L_\mu\|=\epsilon$. By definition, there exists $\{f_j\}\subset U$ such that $|L_\mu(f_j)|\geq\epsilon-1/j$. By compactness of $U$, there exists a subsequence $\{f_{j_k}\}\subset U$ and $f\in U$ such that $f_{j_k}\to f$ uniformly. We immediately have $|L_\mu(f)|\leq \|f\|_\infty\epsilon\leq \epsilon$. For the reverse inequality, as $k\to\infty$,
	\[
	|L_\mu(f)|
	\geq |L_\mu(f_{j_k})|-|L_\mu(f-f_{j_k})|
	\geq \epsilon-\frac{1}{j_k}-\|\mu\|\|f-f_{j_k}\|_\infty
	\to \epsilon. 
	\]
	This proves that $|L_\mu(f)|=\epsilon$. 
	
	\item
	There exists $f\in U$ such that $|L_\mu(f)|=\epsilon$. This implies $e^{i\theta}L_\mu(f)=\epsilon$ for some $\theta\in\R$. Hence, let $\varphi=e^{i\theta}f\in U$, and thus, $L_\mu(\varphi)=\epsilon$.
	
	\item
	Let $\nu\in\E$. Since $\varphi\in U\subset C(\T^d;\Lambda)$, we have
	\[
	\<\varphi,\nu\>
	=\<\varphi,\mu\>
	=L_\mu(\varphi)
	=\epsilon
	=\|\nu\|.
	\]
	Since $\|\varphi\|_\infty\leq 1$ and $\<\varphi,\nu\>=\|\nu\|$, we must have $\varphi=\sign(\nu)$ $\nu$-a.e. 
	
	Since $|\varphi|=|\sign(\nu)|=1$ $\nu$-a.e. and $\nu$ is a Radon measure, 
	\[
	\supp(\nu)
	\subset \overline{\{x\in\T^d\colon |\varphi(x)|=1\}}
	=\{x\in\T^d\colon |\varphi(x)|=1\}.
	\]
	The last equality holds because the inverse image of the closed set $\{1\}$ under the continuous function $|\varphi|$ is closed.
	\end{enumerate}

\end{proof}

\subsection{A characterization of $\Gamma$}
\label{section Gamma}

Proposition \ref{prop1}f,g show that there exists $\varphi\in U$ such that the minimal extrapolations are supported in the closed set $S=\{x\in\T^d\colon |\varphi(x)|=1\}$. By itself, this statement is not useful because if $|\varphi|\equiv 1$, then $S=\T^d$. However, Proposition \ref{prop2} characterizes the case that $|\varphi|\equiv 1$ in terms of the set $\Gamma$. Intuitively, $\#\Gamma$ represents the number of ``bad" functions in $U$ that interpolate the signs of the minimal extrapolations. While it is desirable to have $\Gamma=\emptyset$, perhaps surprisingly, we can make strong statements when $\#\Gamma$ is large.

\begin{proposition}
	\label{prop2}
	Let $\mu\in M(\T^d)$ and $\Lambda\subset\Z^d$ be a finite subset. 
	\begin{enumerate}[(a)]
		\item
		Suppose $\varphi\in U$, $\<\varphi,\mu\>=\epsilon$, and $|\varphi|\equiv 1$. Then, $\Gamma\not=\emptyset$.
		
		\item
		For each $m\in\Z^d$, define $\alpha_m\in\R/\Z$ by the formula $e^{-2\pi i\alpha_m}\hat\mu(m)=|\hat\mu(m)|$. If $m\in\Gamma$, then  
		\[
		\forall\nu\in\E, \quad \sign(\nu)(x)=e^{2\pi i\alpha_m}e^{2\pi im\cdot x} \quad \nu\text{-a.e.}
		\]
	\end{enumerate}
\end{proposition}

\begin{proof}
	\indent
	\begin{enumerate}[(a)]
		
		\item
		Since $\varphi\in U$ and $|\varphi|\equiv 1$, we must have $\varphi=e^{2\pi i\theta}e^{2\pi im\cdot x}$ for some $m\in\Lambda$ and $\theta\in\R$. Then, we have
		\[
		\epsilon=\<\varphi,\mu\>=e^{-2\pi i\theta}\ \overline{\hat\mu(m)},
		\]
		which shows that $m\in\Gamma$. 
		
		\item
		Suppose $m\in\Gamma$ and $\nu\in\E$. Then
		\[
		\int_{\T^d} e^{-2\pi i\alpha_m}e^{-2\pi im\cdot x}\ d\nu(x)
		=e^{-2\pi i\alpha_m}\hat\nu(m)
		=e^{-2\pi i\alpha_m}\hat\mu(m)
		=|\hat\mu(m)|
		=\epsilon
		=\|\nu\|.
		\]
		This shows $\sign(\nu)(x)=e^{2\pi i\alpha_m}e^{2\pi im\cdot x}$ $\nu$-a.e.  
	\end{enumerate}
\end{proof}

\subsection{An analogue of Beurling's theorem}
\label{section Beurling}

We are ready to prove Theorem \ref{thm BL}. 

\begin{proof}
	\indent
	
	\begin{enumerate}[(a)]
		\item
		By Proposition \ref{prop1}, there exists $\varphi\in U$ such that $\<\varphi,\mu\>=\epsilon$, and each minimal extrapolation is supported in the closed set
		\[
		S=\{x\in\T^d\colon |\varphi(x)|=1\}.
		\]
		Note that $|\varphi|\not\equiv 1$. In fact, if $|\varphi|\equiv 1$, then Proposition \ref{prop2}a implies $\Gamma\not=\emptyset$, which is a contradiction. We have $\varphi(x)=\sum_{m\in\Lambda} a_me^{2\pi im\cdot x}$ for some $a_m\in\C$. Consider the function 
		\begin{equation}
		\label{Phi}
		\Phi(x)
		=1-|\varphi(x)|^2
		=1-\sum_{m\in\Lambda}\sum_{n\in\Lambda} a_m\overline{a_n}e^{2\pi i(m-n)\cdot x}.
		\end{equation}
		Then, the minimal extrapolations are supported in the closed set
		\[
		S=\{x\in\T^d\colon \Phi(x)=0\}. 
		\]
		Note that $\Phi\not\equiv 0$ because $|\varphi|\not\equiv 1$. Since $\Phi$ is a non-trivial real-analytic function, $S$ is a set of $d$-dimensional Lebesgue measure zero. In particular, if $d=1$, then $S$ is a finite set of points.
		
		\item
		
		Let $m,n\in\Gamma$. There exist $\alpha_m,\alpha_n\in\R/\Z$ defined in Proposition \ref{prop2}b, such that
		\[
		\forall\nu\in\E,\quad 
		\sign(\nu)(x)
		=e^{2\pi i\alpha_m}e^{2\pi im\cdot x}=e^{2\pi i\alpha_n}e^{2\pi in\cdot x}\quad
		\nu\text{-a.e.}
		\]
		Set $\alpha_{m,n}=\alpha_m-\alpha_n\in\R/\Z$. Then each minimal extrapolation is supported in
		\[
		S_{m,n}=\{x\in\T^d\colon x\cdot(m-n)+\alpha_{m,n}\in\Z\}.
		\]
		Thus, \emph{each minimal extrapolation is supported in the set}
		\[
		S
		=\bigcap_{\substack{m,n\in\Gamma \\ m\not=n}} S_{m,n}
		=\bigcap_{\substack{m,n\in\Gamma \\ m\not=n}} \{x\in\T^d\colon x\cdot(m-n)+\alpha_{m,n}\in\Z\}.
		\]
		Note that $e^{2\pi i\alpha_{m,n}}=\hat\mu(m)/\hat\mu(n)$ because 
		\[
		e^{-2\pi i\alpha_m}\hat\mu(m)
		=|\hat\mu(m)|
		=\epsilon
		=|\hat\mu(n)|
		=e^{-2\pi i\alpha_n}\hat\mu(n).
		\]
		
		Suppose $\{p_1,\dots,p_d\}$ satisfies the hypothesis. By the support assertion that we just proved, there exists $\beta=(\beta_1,\beta_2,\dots,\beta_d)\in\T^d$ such that every minimal extrapolation is supported in
		\[
		S=\bigcap_{j=1}^d \{x\in\T^d\colon x\cdot p_j+\beta_j\in\Z\}.
		\]
		
		Let us explain the geometry of the situation before we proceed with the proof that $S$ is a lattice. Note that $\{x\in\T^d\colon x\cdot p_j+\beta_j\in\Z\}$ is a family of parallel and periodically spaced hyperplanes. Since the vectors, $p_1,p_2,\dots,p_d$, are assumed to be linearly independent, one family of hyperplanes is not parallel to any other family of hyperplanes. Hence, the intersection of $d$ non-parallel and periodically spaced hyperplanes is a lattice, see Figure \ref{fig1} for an illustration. 
		
		For the rigorous proof of this observation, first note that
		\[
		S=\{x\in\T^d\colon Px+\beta\in\Z^d\},
		\]
		where $P=(p_1,p_2,\dots,p_d)^t\in\Z^{d\times d}$ is invertible because its rows are linearly independent. Let $x_0\in\R^d$ be the solution to $Px+\beta=0$, and let $q_j\in\Q^d$ be the solution to $Px=e_j$, where $e_j$ is the standard basis vector for $\R^d$. Then $p_j\cdot q_k=\delta_{j,k}$, and $S$ is generated by the point $x_0$ and the lattice vectors, $q_1,q_2,\dots,q_d$. 
	\end{enumerate}
	
\end{proof}

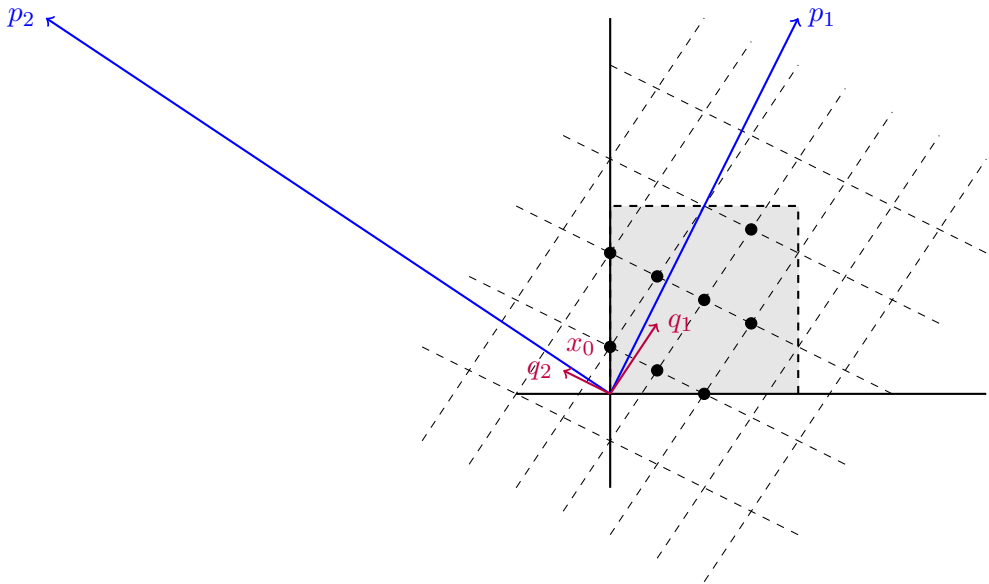
\begin{figure}[h]
	\centering
	\begin{tikzpicture}[scale=2.5]
	\draw[style=thick,dashed,fill=gray!20] (1,0)--(1,1)--(0,1)--(0,0);
	\draw[style=thick] (-0.5,0) -- (2,0);
	\draw[style=thick] (0,-0.5)--(0,2);
	\draw [style=dashed](-1,.25)--(1,-.75);
	\draw [style=dashed](-.75,.625)--(1.25,-.375);
	\draw [style=dashed](-.5,1)--(1.5,0);
	\draw [style=dashed](-.25,1.375)--(1.75,.375);
	\draw [style=dashed](-.0,1.75)--(2,.75);
	\draw [style=dashed](.5,-1)--(2,1.25);
	\draw [style=dashed](.25,-.875)--(1.75,1.375);
	\draw [style=dashed](0,-.75)--(1.5,1.5);
	\draw [style=dashed](-.25,-.625)--(1.25,1.625);
	\draw [style=dashed](-.5,-.5)--(1,1.75);
	\draw [style=dashed](-.75,-.375)--(.75,1.875);
	\draw [style=dashed](-1,-.25)--(.5,2);
	\draw [fill=black] (0,.25) circle (0.03);
	\draw [fill=black] (0,.75) circle (0.03);
	\draw [fill=black] (.25,.125) circle (0.03);
	\draw [fill=black] (.25,.625) circle (0.03);
	\draw [fill=black] (.5,0) circle (0.03);
	\draw [fill=black] (.5,.5) circle (0.03);
	\draw [fill=black] (.75,.375) circle (0.03);
	\draw [fill=black] (.75,.875) circle (0.03);
	\draw [color=purple] (-.025,.25) node[left]{$x_0$};
	\draw [->,style=thick,color=blue](0,0)--(1,2) node[right] {$p_1$};
	\draw [->,style=thick,color=blue] (0,0)--(-3,2) node[left] {$p_2$};
	\draw [->,style=thick,color=purple](0,0)--(.25,.375) node[right]{$q_1$};
	\draw [->,style=thick,color=purple](0,0)--(-.25,.125)node[left]{$q_2$};
	\end{tikzpicture}
	\caption{An illustration of Theorem \ref{thm BL}b, where $d=2$, $p_1=(1,2)$, $p_2=(-3,2)$, $\beta_1=1/2$, $\beta_2=-1/2$, $q_1=(1/4,3/8)$, and $q_2=(-1/4,1/8)$. The family of hyperplanes are the dashed lines, the lattice $S$ is the black dots, and the shaded region is $[0,1)^2$.}
	\label{fig1}
\end{figure}

\begin{remark}
	Example \ref{e5} provides an example where $\#\Gamma=3$ and there exists a singular continuous minimal extrapolation. This demonstrates that the conclusion of Theorem \ref{thm BL}b is optimal. 
\end{remark}

\subsection{The admissibility range for $\epsilon$}
\label{section admissibility}

While the mathematical theory we have developed connects the set $\Gamma$ with the support of the minimal extrapolations, the difficulty of applying the theory is that in general, $\epsilon$ is unknown. However, in many important situations, it is possible to deduce the value of $\epsilon$. We say $[A,B]\subset\R^+$ is an \emph{admissiblity range} for $\epsilon$ provided that $0\leq A\leq\epsilon\leq B$. The following proposition shows that we have $A=\|\hat\mu\|_{\ell^\infty(\Lambda)}$ and $B=\|\mu\|$ as an admissibility range, and $B$ can be improved in certain situations. 

\begin{proposition}
	\label{prop3}
	Let $\mu\in M(\T^d)$ and let $\Lambda\subset\Z^d$ be a finite subset. We have the lower and upper bounds,
	\begin{equation}
		\label{b}
		\|\hat\mu\|_{\ell^\infty(\Lambda)}\leq \epsilon\leq \|\mu\|.
	\end{equation}
	
	Further, if there exists an extrapolation $\nu\in M(\T^d)$ and $\|\nu\|<\|\mu\|$, then
	\begin{equation}
		\label{c}
		\|\hat\mu\|_{\ell^\infty(\Lambda)}\leq \epsilon\leq\|\nu\|<\|\mu\|.
	\end{equation}
\end{proposition}

\begin{proof}

	To see the lower bound for $\epsilon$ in (\ref{b}) and (\ref{c}), let $\sigma$ be a minimal extrapolation. Then,
	\[
	\sup_{m\in\Lambda}|\hat\mu(m)|
	=\sup_{m\in\Lambda}|\hat\sigma(m)|
	\leq \|\sigma\|
	=\epsilon.
	\]
	The upper bounds, $\epsilon\leq\|\mu\|$ and $\epsilon\leq\|\nu\|$ in (\ref{b}) and (\ref{c}), follow by definition of $\epsilon$.
\end{proof}

\begin{figure}[h]
	\centering
	\begin{tikzpicture}[scale=1]
	\path [fill=gray!20] (-3,2) rectangle (3,4);
	\path [fill=gray!40] (-3,2) rectangle (3,3.2);
	\draw [style=dashed] (-3.5,2)--(3.5,2) node [right] {$\|\hat\mu\|_{\ell^\infty(\Lambda)}$};
	\draw [style=dashed] (-3.5,4)--(3.5,4) node [right] {$\|\mu\|$};
	\draw [style=dashed] (-3.5,3.2)--(3.5,3.2) node [right] {$\|\nu\|$};
	\draw[style=thick] (-3.5,0) -- (3.5,0);
	\foreach \x in {-3,...,3}
	\draw (\x, 0.1) -- (\x, -0.1) node [below] {\x};
	\draw (0,-0.7) node [below] {$\Lambda$};
	\draw [fill=black] (-3,1.8) circle (0.05);
	\draw [fill=black] (-2,1.6) circle (0.05);
	\draw [fill=black] (-1,1.8) circle (0.05);
	\draw [fill=black] (0,2) circle (0.05);
	\draw [fill=black] (1,1.6) circle (0.05);
	\draw [fill=black] (2,1.7) circle (0.05);
	\draw [fill=black] (3,2) circle (0.05);
	\end{tikzpicture}
	\caption{The black points represent the values of $|\hat\mu|$ on $\Lambda$. The union of the light and dark gray regions is the admissibility range $[\|\hat\mu\|_{\ell^\infty(\Lambda)},\|\mu\|]$, whereas the dark gray region is an improved admissibility range $[\|\hat\mu\|_{\ell^\infty(\Lambda)},\|\nu\|]$.}
	\label{fig2}
\end{figure}
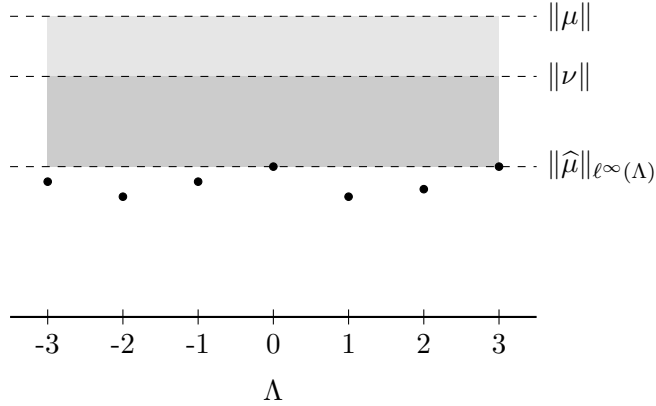

\begin{remark}
	By tightening the admissibility range for $\epsilon$, we can deduce information about the minimal extrapolations. The simplest case is when $\epsilon=\|\hat\mu\|_{\ell^\infty(\Lambda)}=\|\mu\|$, see Examples \ref{e1} and \ref{e2}. The next simplest case is when there exists an extrapolation $\nu$, such that $\epsilon=\|\hat\mu\|_{\ell^\infty(\Lambda)}=\|\nu\|$, see Example \ref{e3}. A more complicated case is when $\|\hat\mu\|_{\ell^\infty(\Lambda)}<\epsilon$, see Example \ref{e4}.
\end{remark}

Cand\`{e}s and Fernandez-Granda \cite{candes2014towards} focused entirely on the case that $\|\mu\|=\epsilon$ because this is a necessary condition to super-resolve $\mu$. In contrast to their result, our theory is better suited for situations when $\|\hat\mu\|_{\ell^\infty}=\epsilon$. This is because Theorem \ref{thm BL}b is strongest when $|\hat\mu(m)|=\epsilon$ for many points $m\in\Lambda$, i.e., when $\#\Gamma$ is large. Hence, our theory is better suited for deducing the impossibility of super-resolution reconstruction. Remarkably though, the following proposition shows that if Algorithm \ref{alg CFG} fails, then $\epsilon=\|\hat\mu\|_{\ell^\infty}$, and since we are given $\hat\mu\mid_\Lambda$, we immediately get the exact value of $\epsilon$. Recall that the algorithm fails if the function $\varphi\in U$ computed in the first step of the algorithm satisfies $\<\varphi,\mu\>=\epsilon$ and $|\varphi|\equiv 1$.

\begin{proposition}
	\label{prop4}
	Let $\mu\in M(\T^d)$ and let $\Lambda\subset\Z^d$ be a finite subset. Suppose there exists $\varphi\in U$ such that $|\varphi|\equiv 1$ and $\<\varphi,\mu\>=\epsilon$. Then, $\Gamma\not=\emptyset$ and  $\epsilon=\|\hat\mu\|_{\ell^\infty(\Lambda)}$. 
\end{proposition}

\begin{proof}
	By Proposition \ref{prop2}, we must have $\Gamma\not=\emptyset$. Let $m\in\Gamma$ and by definition of $\Gamma$, we have $|\hat\mu(m)|=\epsilon$. This combined with the lower bound in Proposition \ref{prop3} proves that $\epsilon=\|\hat\mu\|_{\ell^\infty(\Lambda)}$. 
\end{proof}


\subsection{On uniqueness and non-uniqueness of minimal extrapolations}
\label{section uniqueness}

Theorem \ref{thm BL} provides sufficient conditions for the minimal extrapolations to be supported in a discrete set, but it does not provide sufficient conditions for uniqueness. Since a family of discrete measures supported on a common set behaves essentially like a vector, we use basic linear algebra to address the question of uniqueness when the minimal extrapolations are supported in a discrete set. 

\begin{proposition}
	\label{uniqueness}
	Let $\mu\in M(\T^d)$ and let $\Lambda=\{m_1,m_2,\dots,m_J\}\subset\Z^d$ be a finite subset. Suppose there exists a finite set, $\{x_k\in\T^d\colon k=1,2,\dots,K\}$, such that each minimal extrapolation is supported in this set. Suppose that the matrix,
	\begin{equation}
	\label{matrix}
	E(m_1,\dots,m_J;x_1,\dots,x_K)=
	\begin{pmatrix}
	e^{-2\pi im_1\cdot x_1} &e^{-2\pi im_1\cdot x_2} &\cdots &e^{-2\pi im_1\cdot x_K} \\
	e^{-2\pi im_2\cdot x_1} &e^{-2\pi im_2\cdot x_2} &\cdots &e^{-2\pi im_2\cdot x_K} \\
	\vdots &\vdots & &\vdots \\
	e^{-2\pi im_J\cdot x_1} &e^{-2\pi im_J\cdot x_2} &\cdots &e^{-2\pi im_J\cdot x_K}
	\end{pmatrix},
	\end{equation}
	has full column rank (this can only occur if $J\geq K$). Then, the minimal extrapolation is unique. 
\end{proposition}

\begin{proof}
	Let $\nu$ be the difference of any two minimal extrapolations. Then $\nu$ is also supported in $\{x_k\colon k=1,2,\dots, K\}$ and it is of the form $\nu=\sum_{k=1}^K a_k \delta_{x_k}$. Since $\hat\nu=0$ on $\Lambda$, we have $0=\hat\nu(m_j)=\sum_{k=1}^K a_k e^{-2\pi im_j\cdot x_k}$ for $j=1,2,\dots,J$, which is equivalent to the linear system $Ea=0$, where $a=(a_1,\dots,a_K)\in\C^K$. By assumption, $E$ has full column rank. This implies $a=0$. 
\end{proof}

To finish this subsection, we address the situation when $\#\Gamma=1$, i.e., the missing case of Theorem \ref{thm BL}. We say a measure $\mu\in M(\T^d)$ is a \emph{positive measure} if $\mu(A)\geq 0$ for all Borel sets $A\subset\T^d$. A sequence $\{a_m\in\C\colon m\in\Z^d\}$ is a \emph{positive-definite sequence} if for all sequences $\{b_m\in\C\colon m\in\Z^d\}$ of finite support, we have
\[
\sum_{m,n\in\Z^d} a_{m-n}b_m\overline{b_n}\geq 0.
\]
For $y\in\T^d$, let $T_y\colon M(\T^d)\to M(\T^d)$ be the translation operator defined by $T_y\mu(x)=\mu(x-y)$. For $n\in\Z^d$, let $M_n\colon M(\T^d)\to M(\T^d)$ be the modulation operator defined by $M_n\mu(x)=e^{2\pi in\cdot x}\mu(x)$. 

\begin{proposition}
	\label{prop5}
	Let $\mu\in M(\T^d)$ and let $\Lambda\subset\Z^d$ be a finite subset. Suppose there exists $n\in\Lambda$ such that $(M_{-n}\mu)^\wedge\mid_{\Lambda-n}$ extends to a positive-definite sequence on $\Z^d$. Then, $n\in\Gamma$, and each positive-definite extension of $(M_{-n}\mu)^\wedge\mid_{\Lambda-n}$ corresponds to a positive measure $\nu$, such that $M_n\nu$ is a minimal extrapolation of $\mu$ from $\Lambda$. 
\end{proposition}

\begin{proof}
	Extend $(M_{-n}\mu)^\wedge\mid_{\Lambda-n}$ to a positive-definite sequence. By Herglotz' theorem, there exists a positive measure $\nu\in M(\T^d)$ such that $\hat\nu=(M_{-n}\mu)^\wedge=T_{-n}\hat\mu$ on $\Lambda-n$. Then, $M_n\nu$ is an extrapolation of $\mu$ from $\Lambda$, which shows that 
	\[
	\|\nu\|=\|M_n\nu\|\geq \epsilon(\mu,\Lambda).
	\]
	For the reverse inequality, since $\nu$ is a positive measure, we have $\|\nu\|=\hat\nu(0)$. Then,
	\[
	\|M_n\nu\|
	=\|\nu\|
	=\hat\nu(0)
	=|\hat\nu(0)|
	=|\hat\mu(n)|
	\leq \|\hat\mu\|_{\ell^\infty(\Lambda)}
	\leq \epsilon(\mu,\Lambda),
	\]
	where the last inequality follows by Proposition \ref{prop3}. This shows that $|\hat\mu(n)|=\|\nu\|=\epsilon$, which proves that $n\in\Gamma$ and $M_n\nu$ is a minimal extrapolation of $\mu$ from $\Lambda$.
\end{proof}

\begin{remark}
	Beurling \cite{beurling1989interpolation} and Esseen \cite{esseen1945fourier} essentially proved the analogue of the Proposition for the special case that $n=0$, for $\R$ instead of $\T^d$. Proposition \ref{prop5} generalizes their result to handle situations when $0\not\in\Lambda$. This is important because from the viewpoint of Proposition \ref{symmetries}c, the super-resolution problem is invariant under simultaneous translations of $\hat\mu$ and $\Lambda$, which means that $0\in\Z^d$ is no more special than any other point $n\in\Z^d$. 
\end{remark}

\begin{remark}
	\label{missing}
	Proposition \ref{prop5} suggests that the case $\#\Gamma=1$ is special compared to the cases $\#\Gamma=0$ or $\#\Gamma\geq 2$ because, when $\#\Gamma=1$, there may exist absolutely continuous minimal extrapolations. In Example \ref{e1}, $\#\Gamma=1$, and there exist uncountably many discrete and positive absolutely continuous minimal extrapolations. 
\end{remark}

\begin{remark}
	\label{MEM}
	Suppose that $\mu\in M(\T)$ and $\hat\mu\mid_\Lambda$ can be extended to a positive-definite sequence on $\Z$. In theory, there are an infinite number of such extensions, and one particular method of choosing such an extension is called the \emph{Maximum Entropy Method} (MEM). According to MEM, one extends $\hat\mu\mid_\Lambda$ to the positive-definite sequence $\{a_m\in\C\colon m\in\Z\}$ whose corresponding density function $f\in L^1(\T)$ is the unique maximizer of a specific logarithmic integral associated with the physical notion of entropy, e.g., see \cite[Theorems 3.6.3 and 3.6.6]{benedetto1996harmonic}. MEM is related to spectral estimation methods \cite{childers1978modern}, the maximum likelihood method \cite{childers1978modern}, and moment problems \cite{landau1987maximum}.
\end{remark}

\subsection{Basic properties of minimal extrapolation}
\label{section basic}

Our next goal is to examine the symmetries of the minimal extrapolations. We are interested in the vector space operations, namely, addition of measures and the multiplication of measures by complex constants. We are also interested in the operations that are well-behaved under the Fourier transform on $\T^d$, namely, translation on the torus, modulation by integers, convolution of measures, and product of measures. 

\begin{proposition}
	\label{symmetries}
	Let $\mu\in M(\T^d)$, $\Lambda\subset\Z^d$ be a finite subset, $c\in\C$ non-zero, and $y\in\T^d$. 
	\begin{enumerate}[(a)]
		\item
		Multiplication by constants is bijective: $c\ \epsilon(\mu,\Lambda)=\epsilon(c\mu,\Lambda)$, and   $\nu\in\E(\mu,\Lambda)$ if and only if $c\nu\in\E(c\mu,\Lambda)$. 
		\item
		Translation is bijective: $\epsilon(\mu,\Lambda)=\epsilon(T_y\mu,\Lambda)$, and $\nu\in\E(\mu,\Lambda)$ if and only if $T_y\nu\in\E(T_y\mu,\Lambda)$. 
		\item
		Minimal extrapolation is invariant under simultaneous shifts of \ $\hat\mu$ and $\Lambda$: $\epsilon(\mu,\Lambda)=\epsilon(M_n\mu,\Lambda+n)$, and $\E(\mu,\Lambda)=\E(M_n\mu,\Lambda+n)$.  
		\item
		The product of minimal extrapolations is a minimal extrapolation for the product: For $j=1,2$, let $\mu_j\in M(\T^{d_j})$, let $\Lambda_j\subset\Z^{d_j}$ be a finite subset, and let $\nu_j\in\E(\mu_j,\Lambda_j)$. Then $\epsilon(\mu_1,\Lambda_1)\epsilon(\mu_2,\Lambda_2)=\epsilon(\mu_1\times\mu_2,\Lambda_1\times\Lambda_2)$, and $\nu_1\times\nu_2\in\E(\mu_1\times\mu_2,\Lambda_1\times\Lambda_2)$. 
	\end{enumerate}
\end{proposition}

\begin{proof}
	\indent
	\begin{enumerate}[(a)]
	\item
	If $\nu\in\E(\mu,\Lambda)$, then $c\nu$ is an extrapolation of $c\mu$ from $\Lambda$. Suppose $c\nu\not\in\E(c\mu,\Lambda)$. Then there exists $\sigma$ such that $\hat\sigma=c\hat\mu$ on $\Lambda$ and $\|\sigma\|<\|c\nu\|$. Thus, $\hat\sigma/c=\hat\mu$ on $\Lambda$ and $\|\sigma/c\|<\|\nu\|$, and this contradicts the assumption that $\nu\in\E(\mu,\Lambda)$. The converse follows by a similar argument. 
	
	\item
	If $\nu\in\E(\mu,\Lambda)$, then $T_y\nu$ is an extrapolation of $T_y\mu$ from $\Lambda$. Suppose $T_y\nu\not\in\E(T_y\mu,\Lambda)$. Then there exists $\sigma$ such that $\hat\sigma=(T_y\mu)^\wedge$ on $\Lambda$ and $\|\sigma\|<\|T_y\nu\|=\|\nu\|$. Then $(T_{-y}\sigma)^\wedge=\hat\mu$ on $\Lambda$ and $\|T_{-y}\sigma\|=\|\sigma\|<\|\nu\|$, which contradicts the assumption that $\nu\in\E(\mu,\Lambda)$. The converse follows by a similar argument. 
	
	\item
	If $f\in U(\T^d;\Lambda)$, then $M_nf\in U(\T^d;\Lambda+n)$. By Parseval's theorem,
	\[
	\<f,\mu\>
	=\sum_{m\in\Lambda} \hat f(m)\overline{\hat\mu(m)}
	=\sum_{m\in\Lambda+n} (M_nf)^\wedge(m) \overline{(M_n\mu)^\wedge(m)}
	=\<M_nf,M_n\mu\>.
	\]
	Using Proposition \ref{prop1}e, we see that $\epsilon(\mu,\Lambda)=\epsilon(M_n\mu,\Lambda+n)$. 
	
	If $\nu\in\E(\mu,\Lambda)$, then $M_n\nu$ is an extrapolation of $M_n\mu$ from $\Lambda+n$, and $\|M_n\|=\|\nu\|=\epsilon(\mu,\Lambda)=\epsilon(M_n\mu,\Lambda+n)$. The converse follows similarly. 
	
	\item
	For convenience, let $\mu=\mu_1\times\mu_2$, $\nu=\nu_1\times\nu_2$, $\Lambda=\Lambda_1\times\Lambda_2$, $\epsilon_j=\epsilon(\mu_j,\Lambda_j)$, and $\epsilon=\epsilon(\mu,\Lambda)$. Since $\nu$ is an extrapolation of $\mu$, by Proposition \ref{prop3}, we  have 
	\[
	\epsilon\leq \|\nu\|=\|\nu_1\|\|\nu_2\|=\epsilon_1\epsilon_2.
	\]
	To see the reverse inequality, we use Proposition \ref{prop1}f, and see that there exist $\varphi_j\in U(\T^{d_j};\Lambda_j)$ such that $\epsilon_j=\<\varphi_j,\mu_j\>$, for $j=1,2$. Let $\varphi=\varphi_1\otimes\varphi_2$ and observe that $\varphi\in U(\T^d;\Lambda)$. By Proposition \ref{prop1}e, we have
	\[
	\epsilon
	=\max_{f\in U(\T^d;\Lambda)} |\<f,\mu\>|
	\geq |\<\varphi,\mu\>|
	\geq \<\varphi,\mu\>
	=\<\varphi_1,\mu_1\>\<\varphi_2,\mu_2\>
	=\epsilon_1\epsilon_2.
	\]
	This shows that $\epsilon=\epsilon_1\epsilon_2$, and, since $\|\nu\|=\epsilon_1\epsilon_2$, we conclude that $\nu\in\E(\mu,\Lambda)$. 
		
	\end{enumerate}
\end{proof}

While minimal extrapolation is well-behaved under translation, it not well behaved under modulation. This is because the Fourier transform of modulation is translation, and so $\hat\mu\mid_\Lambda$ and $(M_n\mu)^\wedge\mid_\Lambda$ are, in general, not equal. In contrast, the Fourier transform of translation is modulation, and so $\hat\mu\mid_\Lambda$ and $(T_y\mu)^\wedge\mid_\Lambda$ only differ by a phase factor. We shall prove these statements in Proposition \ref{false}.

\begin{proposition}
	\label{false}
	\indent
	\begin{enumerate}[(a)]
	\item
	For $j=1,2$, there exist $\mu_j\in M(\T)$, a finite subset $\Lambda\subset\Z$, and $\nu_j\in\E(\mu_j,\Lambda)$, such that $\nu_1+\nu_2\not\in\E(\mu_1+\mu_2,\Lambda)$.
	\item
	There exist $\mu\in M(\T)$, a finite subset $\Lambda\subset\Z$, $\nu\in\E(\mu,\Lambda)$, and $n\in\Z$, such that $M_n\nu\not\in\E(M_n\mu,\Lambda)$. 
	\item
	For $j=1,2$, there exist $\mu_j\in M(\T)$, a finite subset $\Lambda\subset\Z$, and $\nu_j\in\E(\mu_j,\Lambda)$, such that $\nu_1*\nu_2\not\in\E(\mu_1*\mu_2,\Lambda)$.

	\end{enumerate}
\end{proposition}

\begin{proof}
	\indent
	\begin{enumerate}[(a)]
	\item
	Let $\mu_1=\delta_0+\delta_{1/2}$, $\mu_2=-\delta_0-\delta_{1/2}$, and $\Lambda=\{-1,0,1\}$. By Example \ref{e1}, we have $\nu_1=\delta_0+\delta_{1/2}\in\E(\mu_1,\Lambda)$, and $\nu_2=-\delta_{1/4}-\delta_{3/4}\in\E(\mu_2,\Lambda)$. Then, $\mu_1+\mu_2=0$, and so $\epsilon(\mu_1+\mu_2,\Lambda)=0$. However, $\nu_1+\nu_2\not\in\E(\mu_1+\mu_2,\Lambda)$ because $\|\nu_1+\nu_2\|=\|\delta_0-\delta_{1/4}+\delta_{1/2}-\delta_{3/4}\|>0$. 
		
	\item
	Let $\mu=\delta_0+\delta_{1/2}$, $\Lambda=\{-1,0,1\}$, and $n=-1$. By Example \ref{e1}, we have $\nu=\delta_{1/4}+\delta_{3/2}\in\E(\mu,\Lambda)$. However, $M_{-1}\mu=\delta_0-\delta_{1/2}$, and by Example \ref{e2}, $\E(M_{-1}\mu,\Lambda)=\{\delta_0-\delta_{1/2}\}$. Thus, $M_{-1}\nu\not\in\E(M_{-1}\mu,\Lambda)$.
		
	\item
	Let $\mu_1=\delta_0+\delta_{1/2}$, $\mu_2=\delta_0-\delta_{1/2}$, and $\Lambda=\{-1,0,1\}$. Then $(\mu_1*\mu_2)^\wedge=0$ on $\Lambda$, which implies $\epsilon(\mu_1*\mu_2,\Lambda)=0$. By Examples \ref{e1} and \ref{e2}, $\nu_1=\mu_1\in\E(\mu_1,\Lambda)$ and $\nu_2=\mu_2\in\E(\mu_2,\Lambda)$. However, $\nu_1*\nu_2\not\in\E(\mu_1*\mu_2,\Lambda)$ because $\|\nu_1*\nu_2\|=\|\delta_0-\delta_1\|>0$. 
		
	\end{enumerate}
\end{proof}

\section{Examples}
\label{examples}

\subsection{Discrete measures}

There are several reasons why we are interested in computing the minimal extrapolations of discrete measures. They are the simplest types of measures, and so, their minimal extrapolations can be computed rather easily. By Theorem \ref{thm BL}, the minimal extrapolations of a non-discrete measure are sometimes discrete measures, so they appear naturally in our analysis. As discussed in Section \ref{section approach}, examples of $\mu$ that have minimal extrapolations supported in a lattice can be interpreted in the context of deterministic compressed sensing. Examples \ref{e1}-\ref{e4} and Example \ref{e5} can be written in the context of compressed sensing.

\begin{remark}
	In view of Proposition \ref{symmetries}\ a,b, and without loss of generality, we can assume any discrete measure $\mu=\sum_{k=1}^\infty a_k\delta_{x_k}\in M(\T^d)$, where $\sum_{k=1}^\infty |a_k|<\infty$, can be written as $\mu=\delta_0+\sum_{k=2}^\infty a_k'\delta_{x_k'}\in M(\T^d)$, where $\sum_{k=2}^\infty |a_k'|<\infty$. 
\end{remark}

\begin{example}
	\label{e1}
	Let $\mu=\delta_0+\delta_{1/2}$, and $\Lambda=\{-1,0,1\}$. We have $\hat\mu(0)=2$, and $\hat\mu(\pm 1)=0$. Clearly $\|\hat\mu\|_{\ell^\infty(\Lambda)}=\|\mu\|=2$. By Proposition \ref{prop3}, $\epsilon=\|\hat\mu\|_{\ell^\infty(\Lambda)}=\|\mu\|=2$, which implies $\mu\in\E$. 
	
	Further, there is an uncountable number of discrete minimal extrapolations. To see this, for each $y\in\T$ and any integer $K\geq 2$, define 
	\[
	\nu_{y,K}=\frac{2}{K}\sum_{k=0}^{K-1}\delta_{y+\frac{k}{K}}.
	\] 
	A straightforward calculation shows that $\nu_{y,K}$ is an extrapolation and that $\|\nu_{y,K}\|=\epsilon$. 
	
	Also, we can construct positive absolutely continuous minimal extrapolations. One example is the constant function $f\equiv 2$ on $\T$. For other examples, let $N\geq 2$ and let $F_N\in C^\infty(\T)$ be the Fej\'{e}r kernel,
	\[
	F_N(x)=\sum_{n=-N}^N \(1-\frac{|n|}{N+1}\)e^{2\pi inx}.
	\]
	For any $c>0$ such that $c\leq (2N+2)/(3N+1)$, extend $\hat\mu\mid_\Lambda$ to the sequence $\{(a_{N,c})_m\colon m\in\Z\}$, where 
	\[
	(a_{N,c})_m=
	\begin{cases}
	2 &m=0, \\
	c\(1-\frac{|m|}{N+1}\) &2\leq|m|\leq N, \\
	0 &\text{otherwise}.
	\end{cases}
	\]
	Consider the real-valued function
	\begin{align*}
	f_{N,c}(x)
	&=2+\sum_{n=-N}^{-2} (a_{N,c})_n e^{2\pi inx}+\sum_{n=2}^N (a_{N,c})_ne^{2\pi inx} \\
	&=2+c\sum_{n=-N}^{-2} \(1-\frac{|n|}{N+1}\) e^{2\pi inx}+c\sum_{n=2}^N \(1-\frac{|n|}{N+1}\)e^{2\pi inx}. 
	\end{align*}
	We check that $\hat{f_{N,c}}(m)=(a_{N,c})_m$ for all $m\in\Z$, which implies $f_{N,c}$ is an extrapolation of $\mu$. Using the upper bound on $c$, we have, for all $x\in\T$, that
	\[
	\quad 2\geq c+2c\(1-\frac{1}{N+1}\)\cos(2\pi x)=c+c\(1-\frac{1}{N+1}\)e^{2\pi ix}+c\(1-\frac{1}{N+1}\)e^{-2\pi ix}.
	\]
	Using this inequality, and the definitions of $F_N$ and $f_{N,c}$, we have
	\[
	f_{N,c}(x)
	\geq c F_N
	\geq 0.
	\]
	Since $f_{N,c}\geq 0$, we also have
	\[
	\|f_{N,c}\|_1
	=\int_\T f_{N,c}(x)\ dx
	=2+\int_\T \(\sum_{n=-N}^{-2} (a_{N,c})_n e^{2\pi inx}+\sum_{n=2}^N (a_{N,c})_ne^{2\pi inx}\)\ dx
	=2
	=\epsilon.
	\]
	Thus, for any $N\geq 2$ and $c\leq (2N+2)/(3N+1)$, $f_{N,c}$ is a \emph{positive absolutely continuous} minimal extrapolation. Hence, we have constructed an uncountable number of positive absolutely continuous minimal extrapolations. 
	
\end{example}

\begin{example}
	\label{e2}
	Let $\mu=\delta_0-\delta_{1/2}$, and $\Lambda=\{-1,0,1\}$. We have $\hat\mu(0)=0$, and $\hat\mu(\pm 1)=2$. Further, we have $\|\hat\mu\|_{\ell^\infty(\Lambda)}=\|\mu\|=2$, so that by Proposition \ref{prop3}, we have $\epsilon=\|\hat\mu\|_{\ell^\infty(\Lambda)}=\|\mu\|=2$ and $\mu\in\E$. 
	
	Consequently, $\Gamma=\{-1,1\}$. By Theorem \ref{thm BL}b, there exists $\alpha_{-1,1}\in\R/\Z$ satisfying
	\[
	e^{2\pi i\alpha_{-1,1}}
	=\frac{\hat\mu(-1)}{\hat\mu(1)}
	=1,
	\]
	and the minimal extrapolations are supported in the set $\{x\in\T\colon 2x\in\Z\}=\{0,1/2\}$. This implies each $\nu\in\E$ is discrete and can be written as $\nu=a_1\delta_0+a_2\delta_{1/2}$. In theory, $a_1,a_2$ depend on $\nu$, so we cannot conclude uniqueness yet. 
	
	The matrix $E$ from (\ref{matrix}) is
	\[
	E(-1,0,1;0,1/2)=
	\begin{pmatrix}
	1 &-1 \\ 1 &1 \\ 1 &-1
	\end{pmatrix}.
	\]
	Clearly, $E$ has full column rank, so, by Proposition \ref{uniqueness}, $\mu$ is the unique minimal extrapolation.  Thus, super-resolution reconstruction of $\mu$ from $\Lambda$ is possible.
\end{example}

\begin{example}
	\label{e3}
	Let $\mu=\delta_0-\delta_{1/4}$, and let $\Lambda=\{-1,0,1\}$. We have $\hat\mu(\pm 1)=1\pm i=\sqrt 2e^{\pm\pi i/4}$, and $\hat\mu(0)=0$. Note that $\|\hat\mu\|_{\ell^\infty(\Lambda)}=\sqrt 2<2=\|\mu\|$, which shows that $\sqrt 2\leq \epsilon\leq 2$. We claim that $\epsilon=\sqrt 2$. To see this, consider $\nu=(-\delta_{3/8}+\delta_{7/8})/\sqrt 2$. We verify that $\|\nu\|=\sqrt 2$ and that $\nu$ is an extrapolation. By Proposition \ref{prop3}, $\epsilon=\sqrt 2$ and $\nu\in\E$. This also implies $\mu\not\in\E$, and so super-resolution reconstruction of $\mu$ from $\Lambda$ is impossible.
	
	We claim that $\nu$ is the unique minimal extrapolation. The matrix $E$ from (\ref{matrix}) is 
	\[
	E(-1,0,1;3/8,7/8)=
	\begin{pmatrix}
	e^{2\pi i 3/8} &e^{2\pi i 7/8} \\
	1 &1 \\
	e^{-2\pi i3/8} &e^{-2\pi i7/8}
	\end{pmatrix},
	\]
	which we observe to have full column rank. By Proposition \ref{uniqueness}, we conclude that $\nu$ is the unique minimal extrapolation. Thus, super-resolution reconstruction of $\nu$ from $\Lambda$ is possible.
	
	We explain the derivation of $\nu$. We guess that $\epsilon=\sqrt 2$ and see what Theorem \ref{thm BL}b implies. Under this assumption that $\epsilon=\sqrt 2$, we have $\Gamma=\{-1,1\}$. By Theorem \ref{thm BL}b, there exists $\alpha_{-1,1}\in\R/\Z$ satisfying 
	\[
	e^{2\pi i\alpha_{1,-1}}
	=\frac{\hat\mu(1)}{\hat\mu(-1)}
	=e^{\pi i/2},
	\]
	and the minimal extrapolations are supported in $\{x\in\T\colon 2x+1/4\in\Z\}=\{3/8,7/8\}$. Hence, if $\epsilon=\sqrt 2$, then every $\sigma\in\E$ is of the form $\sigma=a_1\delta_{3/8}+a_2\delta_{7/8}$. Thus, by definition of a minimal extrapolation, $\|\sigma\|=\sqrt 2$ and $\hat\mu=\hat\sigma$ on $\Lambda$. Using this information, we solve for the coefficients $a_1,a_2$, and compute that $|a_1|=|a_2|=\sqrt 2$, $a_1=-a_2$, and $a_1=-\sqrt 2/2$. Thus, we obtain that $\sigma=(-\delta_{3/8}+\delta_{7/8})/\sqrt 2$. From here, we simply check that $\nu=\sigma$ is, in fact, a minimal extrapolation. 
\end{example}

\begin{example}
	\label{e4}
	Let $\mu=\delta_0+e^{\pi i/3}\delta_{1/3}$ and let $\Lambda=\{-1,0,1\}$. We have $\hat\mu(-1)=0$, $\hat\mu(0)=1+e^{\pi i/3}=\sqrt 3 \ e^{\pi i/6}$, and $\hat\mu(1)=1+e^{-\pi i/3}=\sqrt 3 \ e^{-\pi i/6}$. 
	
	Suppose, for the purpose of obtaining a contradiction, that $\epsilon=\|\hat\mu\|_{\ell^\infty(\Lambda)}=\sqrt 3$. Then $\Gamma=\{0,1\}$. By Theorem \ref{thm BL}b, there is $\alpha_{0,1}\in\R/\Z$ such that 
	\[
	e^{2\pi i\alpha_{0,1}}
	=\frac{\hat\mu(0)}{\hat\mu(1)}
	=e^{\pi i/3},
	\]
	and each $\nu\in\E$ is of the form $\nu=a\delta_{1/6}$ for some $a\in\C$. Then, $|\hat\nu|=|a|$ on $\Z$ and, in particular, $\hat\mu\not=\hat\nu$ on $\Lambda$, which is a contradiction. 
	
	Thus, $\epsilon>\sqrt 3$, i.e., $\Gamma=\emptyset$. Therefore, Theorem \ref{thm BL}a applies, and so there is a finite set $S$ such that $\supp(\nu)\subset S$ for each $\nu\in\E$. In particular, each $\nu\in\E$ is discrete. Hence, we have to solve the optimization problem given in Proposition \ref{prop1}e, which is
	\[
	\epsilon=\max\Big\{ \Big| a\sqrt 3\ e^{\pi i/6}+b\sqrt 3\ e^{-\pi i/6} \Big|\colon \forall x\in\T,\ 
	|ae^{2\pi ix}+b+ce^{-2\pi ix}|\leq 1,\ a,b,c\in\C \Big\}.
	\] 
	This optimization problem can be written as a semi-definite program, see \cite[Corollary 4.1, page 936]{candes2014towards}. After obtaining numerical approximations to the optimizers of this problem, we guess that the the exact optimizers are
	\[
	a=\frac{2}{3\sqrt 3}e^{-\pi i/6},\quad
	b=\frac{4}{3\sqrt 3}e^{\pi i/6},\quad
	c=-\frac{i}{3\sqrt 3}.
	\]
	These values of $a,b,c$ are, in fact, the optimizers because $a\sqrt 3\ e^{\pi i/6}+b\sqrt 3\ e^{-\pi i/6}=2$ and $|ae^{2\pi ix}+b+ce^{-2\pi ix}|\leq 1$ for all $x\in\T$. Thus, $\epsilon=2$ and $\mu\in\E$. Since $|ae^{2\pi ix}+b+ce^{-2\pi ix}|=1$ precisely on $S=\{0,1/3\}$, by Theorem \ref{thm BL}a, the minimal extrapolations are supported in $S$. 
	
	The matrix $E$ from (\ref{matrix}) is 
	\[
	E(-1,0,1;0,1/3)=
	\begin{pmatrix}
	1 &e^{-2\pi i/3} \\ 1 &1\\ 1 &e^{2\pi i/3}
	\end{pmatrix}.
	\]
	Since $E$ has full rank, by Proposition \ref{uniqueness}, $\mu$ is the unique minimal extrapolation. Thus, super-resolution reconstruction of $\mu$ from $\Lambda$ is possible. 
\end{example}

The following example illustrates that if $\mu$ is a sum of two Dirac measures, then their supports have to be sufficiently spaced apart in order for super-resolution of $\mu$ to be possible. This shows that, in general, a minimum separation condition is necessary to super-resolve a sum of two Dirac measures, see \cite{candes2014towards}. 

\begin{example}
	\label{e6}
	Let $\mu_y=\delta_0-\delta_y$ for some non-zero $y\in\T^d$ and let $\Lambda\subset\Z^d$ be a finite subset. We claim that if $y$ is sufficiently small depending on $\Lambda$, then $\mu_y\not\in\E(\mu_y,\Lambda)$. Note that $\|\mu_y\|=2$ for any $y\in\T^d$. Let $\eta$ denote the normalized Lebesgue measure on $\T^d$ and define the measures $\nu_y$ by the formula
	\[
	\nu_y(x)=\sum_{m\in\Lambda}\hat{\mu_y}(m)e^{2\pi im\cdot x}\ \eta(x).
	\]
	By construction, $\nu_y$ is an extrapolation of $\mu_y$ because, for each $n\in\Lambda$,
	\[
	\hat{\nu_y}(n)
	=\int_{\T^d}e^{-2\pi in\cdot x}\ d\nu_y(x)
	=\sum_{m\in\Lambda}\hat{\mu_y}(m)\int_{\T^d} e^{-2\pi i(n-m)\cdot x}\ dx
	=\hat{\mu_y}(n).
	\]
	Then, as $y\to 0$, 
	\[
	\|\nu_y\|
	=\int_{\T^d} \Big|\sum_{m\in\Lambda} \hat{\mu_y}(m)e^{2\pi im\cdot x}\Big| \ dx\
	=\int_{\T^d} \Big|\sum_{m\in\Lambda} (1-e^{-2\pi im\cdot y}) e^{2\pi im\cdot x}\Big| \ dx
	\to 0.
	\]
	Thus, we take $y$ sufficiently small so that $\|\nu_y\|<2=\|\mu_y\|$, and, hence, $\mu_y\not\in\E(\mu_y,\Lambda)$. Note that this argument does not contradict Proposition \ref{prop3} because $\|\hat{\mu_y}\|_{\ell^\infty(\Lambda)}\to 0$ as $y\to 0$. Thus, for $y$ sufficiently small, super-resolution reconstruction of $\mu_y$ from $\Lambda$ is impossible. 
\end{example}

\subsection{Singular continuous measures}
The following example is an analogue of Example \ref{e1} for higher dimensions.

\begin{example}
	\label{e5}
	Let $\mu=\delta_{(0,0)}+\delta_{(1/2,1/2)}$, and $\Lambda=\{-1,0,1\}^2\setminus\{(1,-1),(-1,1)\}$. Then,  $\hat\mu(m)=1+e^{-\pi i(m_1+m_2)}$, and, in particular, $\hat\mu(1,1)=\hat\mu(-1,-1)=\hat\mu(0,0)=2$ and $\hat\mu(\pm 1,0)=\hat\mu(0,\pm 1)=0$. We deduce that $\epsilon=\|\mu\|=\|\hat\mu\|_{\ell^\infty(\Lambda)}=2$ from Proposition \ref{prop3}, and so $\mu\in\E$. 
	
	Further, $\Gamma=\{(0,0),(1,1),(-1,-1)\}$, and so $\#\Gamma=3$. According to the definition of $\alpha_{m,n}$ in Theorem \ref{thm BL}b, set $\alpha_{(-1,-1),(0,0)}=\alpha_{(0,0),(1,1)}=\alpha_{(-1,-1),(1,1)}=0$. By the conclusion of Theorem \ref{thm BL}b, the minimal extrapolations are supported in the set $S=S_{(-1,-1),(0,0)}\cap S_{(0,0),(1,1)}\cap S_{(-1,-1),(1,1)}$, where
	\begin{align*}
	S_{(-1,-1),(0,0)}&=\{x\in\T^2\colon x\cdot (-1,-1)\in\Z\}=\{x\in\T^2\colon x_1+x_2\in\Z\}, \\
	S_{(0,0),(1,1)}&=\{x\in\T^2\colon x\cdot (-1,-1)\in\Z\}=\{x\in\T^2\colon x_1+x_2\in\Z\}, \\
	S_{(-1,-1),(1,1)}&=\{x\in\T^2\colon x\cdot (-2,-2)\in\Z\}=\{x\in\T^2\colon 2x_1+2x_2\in\Z\}.
	\end{align*}
	It follows that the minimal extrapolations are supported in
	\[
	S
	=S_{(-1,-1),(0,0)}\cap S_{(0,0),(1,1)}\cap S_{(-1,-1),(1,1)}
	=\{x\in\T^2\colon x_1+x_2=1\}.
	\]
	
	We can construct other discrete minimal extrapolations besides $\mu$. For each $y\in\T$ and for each integer $K\geq 2$, define the measure
	\[
	\nu_{y,K}=\frac{2}{K}\sum_{k=0}^{K-1} \delta_{\big(y+\frac{k}{K},1-y-\frac{k}{K}\big)}.
	\]
	We claim $\nu_{y,K}$ is a minimal extrapolation. We have $\|\nu_{y,K}\|=\epsilon$, and 
	\[
	\hat{\nu_{y,K}}(m)
	=e^{-2\pi i(m_1y-m_2y)}e^{-2\pi im_2}\frac{2}{K}\sum_{k=0}^{K-1} e^{-2\pi i(m_1-m_2)k/K}. 
	\]
	We see that $\hat{\nu_{y,K}}=\hat\mu$ on $\Lambda$, which proves the claim. 
	
	We can also construct continuous singular minimal extrapolations. Let $\sigma=\sqrt 2\sigma_S$, where $\sigma_S$ is the \emph{surface measure} of the Borel set $S$. We readily verify that $\|\sigma\|=\epsilon$ and 
	\[
	\hat\sigma(m)
	=\sqrt 2\int_{\T^2} e^{-2\pi im\cdot x}\ d\sigma_S
	=2e^{-2\pi im_2}\int_0^1 e^{-2\pi i(m_1-m_2)t}\ dt
	=2\delta_{m_1,m_2},
	\] 
	which proves that $\sigma\in\E$. In particular, $S$ is the smallest set that contains the support of all the minimal extrapolations. 
	
	Since $\mu$ is not the unique minimal extrapolation, super-resolution reconstruction of $\mu$ from $\Lambda$ is impossible. 
\end{example}

\begin{example}
	For an integer $q\geq 3$, let $C_q$ be the \emph{middle $1/q$-Cantor set}, which is defined by $C_q=\bigcap_{k=0}^\infty C_{q,k}$, where $C_{q,0}=[0,1]$ and 
	\[
	C_{q,k+1}=\frac{C_{q,k}}{q}\cup \((1-q)+\frac{C_{q,k}}{q}\).
	\]
	Let $F_q\colon[0,1]\to[0,1]$ be the \emph{Cantor-Lebesgue} function on $C_q$, which is defined by the point-wise limit of the sequence $\{F_{q,k}\}$, where $F_{q,0}(x)=x$ and
	\[
	F_{q,k+1}(x)=
	\begin{cases}
	\ \frac{1}{2}F_{q,k}(qx) &0\leq x\leq\frac{1}{q}, \\
	\ \frac{1}{2}			 &\frac{1}{q}\leq x\leq\frac{q-1}{q}, \\
	\ \frac{1}{2}F_{q,k}(qx-(q-1))+\frac{1}{2} &\frac{q-1}{q}\leq x\leq 1.
	\end{cases}
	\]
	By construction, $F_q(0)=0$, $F_q(1)=1$, and $F_q$ is non-decreasing and uniformly continuous on $[0,1]$. Thus, $F_q$ can be uniquely identified with the measure $\sigma_q\in M(\T)$, and $\|\sigma_q\|=1$. Since $F_q'=0$ a.e. and $F_q$ does not have any jump discontinuities, $\sigma_q$ is a continuous singular measure, with zero discrete part. The Fourier coefficients of $\sigma_q$  are
	\[
	\hat{\sigma_q}(m)
	=(-1)^m\prod_{k=1}^\infty \cos(\pi m q^{-k}(1-q)),
	\]
	see \cite[pages 195-196]{zygmund2002trigonometric}. In particular, for any integer $n\geq 1$, we have
	\[
	\hat{\sigma_q}(q^n)
	=(-1)^{q^n}\prod_{k=1}^\infty \cos(\pi q^{-k}(1-q)),
	\]
	which is convergent and independent of $n$. Since $\hat{\sigma_q}(0)=\|\sigma_q\|=1$, we immediately see that $\epsilon=1$ and $\sigma_q\in\E$. Again, we cannot determine whether $\sigma_q$ is the unique minimal extrapolation because Theorem \ref{thm BL} cannot handle the case $\#\Gamma=1$, see Remark \ref{missing}.
\end{example}

\begin{example}
	Let $\sigma_A,\sigma_B\in M(\T^d)$ be the surface measures of the Borel sets $A=\{x\in\T^2\colon x_2=0\}$ and $B=\{x\in\T^2\colon x_2=1/2\}$, respectively. Let $\mu=\sigma_A+\sigma_B$, and $\Lambda=\{-2,-1,\dots,1\}^2$. Then,
	\[
	\hat\mu(m)
	=\int_0^1 e^{2\pi im_1t}\ dt+\int_0^1 e^{2\pi i(m_1t+m_2/2)}\ dt
	=\delta_{m_1,0}+(-1)^{m_2}\delta_{m_1,0}.
	\]
	We immediately see that $\epsilon=\|\hat\mu\|_{\ell^\infty(\Lambda)}=\|\mu\|=2$, which implies $\mu\in\E$. Then, $\Gamma=\{(0,0),(0,2),(0,-2)\}$, and, by Theorem \ref{thm BL}b, the minimal extrapolations are supported in $\{x\in\T^2\colon x_2=0\}\cup\{x\in\T^2\colon x_2=1/2\}$. Determining whether $\mu$ is the unique minimal extrapolation is beyond the theory we have developed herein, and we shall examine this uniqueness problem in \cite{benli2}. 
\end{example}

\section{Acknowledgements}
The first named author gratefully acknowledges the support of ARO Grant W911 NF-15-1-0112, ARO Grant W911 NF-16-1-0008, and DTRA Grant 1-13-1-0015.

\nocite{kahane1970series, rudin1990fourier, au2013generalized, katsaggelos2007super, hille1929remarks, salem1942singular, benedetto1975spectral, slepian1978prolate, chandrasekaran2012convex}
\bibliography{SRBeurlingME}
\bibliographystyle{alpha}

\end{document}